\documentclass[a4paper]{article}
\usepackage{graphicx}
\usepackage{amsmath,amssymb,latexsym}
\usepackage{amsthm}
\usepackage{hyperref}
\usepackage[top=2.5cm, bottom=3cm, left=3cm, right=3cm]{geometry}
\usepackage{subcaption}

\newcommand{\norm}[1]{\left\Vert#1\right\Vert}

\newcommand{\dprod}[2]{\left\langle#1,#2 \right\rangle}
\newcommand{\vol}[1]{\text{Vol}\left(#1 \right)}
\newcommand{\bR}{\displaystyle\mathbb{R}}

\newcommand{\bE}{\displaystyle\mathbb{E}}
\newcommand{\bP}{\displaystyle\mathbb{P}}

\newcommand{\bZ}{\displaystyle\mathbb{Z}}

\newcommand{\proj}{\textrm{Proj}}

\newcommand{\var}{\textrm{Var}}

\newtheorem{theorem}{Theorem}[section]

\newtheorem{prop}[theorem]{Proposition}
\newtheorem{remark}[theorem]{Remark}
\newtheorem{example}[theorem]{Example}

\newtheorem{defi}[theorem]{Definition}
\newtheorem{cor}[theorem]{Corollary}

\newtheorem*{conj*}{Conjecture}

\newtheorem{question}[theorem]{Question}

\renewcommand{\S}{Section }

\begin{document}
\title{Sampling by Intersections with Random Geodesics}
\author{Uri Grupel}
\date{}
\maketitle


\begin{abstract}
	In this paper we compare the different phenomena that occur when intersecting geometric objects with random geodesics on the unit sphere and inside convex bodies. On the high dimensional sphere we see that with probability bounded away from zero, the observed length will deviate from the actual measure by at most a fixed error for any subset, while in convex bodies we can always choose a subset for which the behavior would be close to a zero-one law, as the dimension grows. The result for the sphere is based on an analysis of the Radon transform. Using similar tools we analyze the variance of intersections on the sphere by higher dimensional random subspaces, and on the discrete torus by random arithmetic progressions.
\end{abstract}

\section{Introduction}

The main question of this paper is whether the length of the intersection with a random geodesic curve can represent faithfully the measure of a given set. While we consider this to be a natural question in geometry, there are additional motivations for such an investigation. One such motivation is connected to the efficiency of algorithmic sampling results, such as \textit{hit and run} \cite{lovasz+simonovits93}.

Given a subset $A$ of some ambient space $M$, the length of the intersection of $A$ with a random geodesic curve is related to the probability of escaping the set $A$ by choosing a uniform point on the geodesic. This is related to the notion of \textit{conductance}, which is key in studying the effectiveness of hit and run algorithms (e.g \cite{lovasz+simonovits93,lovasz+vempala04}).

We start with the unit sphere $S^{n-1}=\{x\in\bR^n;\;|x|=1\}$ as our ambient space. In this case, we expand our scope from geodesics to higher dimensional subspaces. Let $A\subseteq S^{n-1}$ be a measurable subset of the $n$-dimensional sphere. Let $H\subseteq\bR^n$ be a random $k$-dimensional linear subspace. We investigate the random variable
\[X=\frac{\sigma_{H}(A\cap H)}{\sigma_{n-1}(A)},\]
where $\sigma_{n-1}$ and $\sigma_H$ are the rotationally invariant probability measures on $S^{n-1}$ and $S^{n-1}\cap H$ respectively.

The case $k=2$ is our original question of intersection of a set with a random geodesic. In \cite{klartag+regev11}, Klartag and Regev used the case of $k=n/2$ in order to give a lower bound on the communication complexity of the \textit{Vector in Subspace Problem}. The connection between the concentration phenomenon of random intersections and communication complexity is through the \textit{rectangle method}. Here, high concentration of $X$ shows that any protocol would require the exchange of many bits in order to distinguish between different states.

In \cite{klartag+regev11} Klartag and Regev showed that when the dimension $k$ is large, the random variable $X$ is highly concentrated around its mean $1$. Their proof consisted of two main steps. First, they dealt with the case where $H$ is a random hyperplane ($k=n-1$). Then they used the result of hyperplanes repeatedly in order to obtain concentration inequalities for $k$ which depends on $n$ linearly.

The result of Klartag and Regev is sharp, but their method proves to be more difficult when the dimension of $H$ is low. Hence we employ a direct analysis of the affect the intersection has, by defining the Radon transform, as we discuss in \S \ref{sec:radon}. In \S \ref{sec:sphere_geodesics} we build upon this analysis and obtain a dimension free estimate on the probability of $X$ diverging from one by a fixed percentage, for random geodesics (the case $k=2$). This result requires us to consider sets of large measure.

\begin{theorem}\label{thm:sphere_geodesics}
	Let $A\subseteq S^{n-1}$ be a measurable subset, such that $\sigma_{n-1}(A)=1/2$ then
	\[\bP_{L}\left(\left|\frac{\sigma_{L}(A\cap L)}{\sigma_{n-1}(A)}-1\right|\geq\frac{1}{2^{1/3}}\right)\leq \frac{1}{2^{1/3}},\]
	where $L$ is a uniformly chosen geodesic curve on $S^{n-1}$ and $\sigma_{L}$ is the uniform probability measure on $S^{n-1}\cap L$.
\end{theorem}

Remark \ref{rmk:sharp} in \S \ref{sec:sphere_geodesics} shows that we cannot improve the probability bound to a bound that tends to zero when the dimension grows.

In \S \ref{sec:convex_geodesics} we show that there is no analogous result for random geodesics inside convex sets. We show, that for any convex body, we can construct a subset of half the volume of the body, such that a random geodesic curve would either miss it or its complement with high probability.

\begin{theorem}\label{thm:convex_geodesics}
	Let $K\subseteq\bR^n$ be a convex body. There exists a set $A\subseteq K$ such that $\vol{A}/\vol{K}=1/2$ and
	\[\bP\left(\frac{\textrm{length}(L\cap A)}{\textrm{length}(L\cap K)}\in\{0,1\}\right)=1-O^*\left(\frac{1}{\sqrt{n}}\right),\]
	where $L=X+\bR\theta$, $X$ and $\theta$ are independent and distributed uniformly in $K$ and $S^{n-1}$ respectively.
\end{theorem}

Here, $O^*$ represent the big O notation up to logarithmic factors. 

The above construction relies on the concentration of measure in convex bodies. It is known that for an isotropic convex body, most of its mass is concentrated in a thin spherical shell. Combining this with the concentrated one dimensional marginals of a uniform random vector on the sphere, we see that a typical random geodesic will miss a neighborhood of the barycenter of the body. This allows us to construct the desired subset on the convex body.

The different phenomena observed in Theorems \ref{thm:sphere_geodesics} and \ref{thm:convex_geodesics}, raise the question: which of the two would occur for random geodesics in different spaces? In \S \ref{sec:finite_fields} we show how the same tools used to analyze random geodesics on the sphere can be used on the discrete tours $\bZ/p\bZ$ where $p$ is prime, and obtain a similar result to Theorem \ref{thm:sphere_geodesics}.

It would be interesting to understand which other ambient spaces $M$ behave similarly to the sphere and the discrete torus, and which behave similarly to convex sets. A more general question is whether we can find a sufficient condition for such phenomena. A possible candidate would be a curvature condition.

\begin{question}
	Does a positive Ricci curvature implies a theorem analogous to Theorem \ref{thm:sphere_geodesics}?
\end{question}

The main tool for understanding both the sphere and the discrete torus is to define the appropriate \textit{Radon transform}, that averages functions on subspaces. In \S \ref{sec:radon} we develop our method to calculate its singular decomposition. The calculations in \S \ref{sec:radon} are not limited to the case of geodesics ($k=2$), and in \S \ref{sec:general_k} we expand our analysis to all $2\leq k\leq n-1$.

One of the consequences of this analysis is a bound on the variance of the random variable $\sigma_{H}(A\cap H)/\sigma_{n-1}(A)$.

\begin{theorem}\label{thm:general_k}
	Let $A\subseteq S^{n-1}$ be a measurable set. Let $2\leq k\leq n-1$, and let $H\subseteq \bR^n$ be a random subspace of dimension $k$. Then,
	\[\textrm{Var}\left(\frac{\sigma_H(A\cap H)}{\sigma_{n-1}(A)}\right)\leq \frac{n-k}{k(n-1)}\left(\frac{1}{\sigma_{n-1}(A)}-1\right).\]
\end{theorem}

The above analysis could possibly lead to a direct proof of the result of Klartag and Regev, or to generalize \cite{grupel17}, where we average the measures of the intersections with a random subspace and its orthogonal complement.

\textit{\textbf{{Acknowledgement.}}} This paper was written under the supervision of Bo'az Klartag whose guidance, support and patience were invaluable.
In addition, I would like to thank Alon Nishry, Liran Rotem, Boaz Slomka, and Ben Cousins for many fruitful and helpful discussions. Some of this work was done while visiting the Mathematical Sciences Research Institute (MSRI) and is supported by the European Research Council (ERC).

\section{The Radon Transform}\label{sec:radon}
In this section we introduce the Radon transform, an integral transform that averages a function along a given subspace. We denote by $G_{n,k}$ the Grassmanian manifold of all $k$~dimensional subspaces of $\bR^n$.

\begin{defi}
	Let $1\leq k \leq n$. The Radon transform $R_k:L^2\left(S^{n-1}\right)\to L^2\left(G_{n,k}\right)$ is defined by
	\[R_k f(E)=\int_{S^{n-1}\cap E}f(x)d\sigma_E(x),\]
	where $\sigma_E$ is the $SO(n)$ invariant Haar probability measure on $S^{n-1}\cap E$.
\end{defi}

We can see that the radon transform gives a functional version of our geometric question. Taking $f$ to be the normalized indicator of a subset $A\subseteq S^{n-1}$, and $E$ to be a random subspace, we obtain
\[R_k f(E)=\int_{S^{n-1}\cap E}\frac{1_A(x)}{\sigma_{n-1}(A)}d\sigma_E(x)=\frac{\sigma_E(A\cap E)}{\sigma_{n-1}(A)}.\]
Hence, understanding the singular values of the radon transform, will give us a clearer picture of the behavior under random intersections.

In this section, we express the singular values of the Radon transform $R_k$ by a one dimensional integral. 
In later sections we analyze this expression in order to understand random intersections on the sphere by $k$~dimensional subspaces.

In order to find the singular values we introduce the conjugate transform of the Radon transform.

\begin{defi}
	Let $1\leq k \leq n$. The conjugate Radon transform
	$R_k^*:L^2\left(G_{n,k}\right) \to L^2\left(S^{n-1}\right)$ is defined by
	\[R_k^* g(\theta)=\int_{\left\{E;\;\theta\in E\right\}}g(E)d\mu_{k,\theta}(E),\]
	Where $\mu_{k,\theta}$ is the Haar probability measure on
	\[\left\{E\in G_{n,k};\;\theta\in E \right\},\]
	invariant under the $SO(n-1)$ action of rotations on $\bR^n$ that fix $\theta$.
\end{defi}

For every $1\leq k\leq n$ we define the operator $S_k=R_k^*R_k$. By definition,
if $\lambda_{k,i}^2$ is an eigenvalue of $S_k$ then $\lambda_{k,i}>0$ is a singular value of $R_k$.

We use the symmetries of $S_k$ in order to show that the spherical harmonics are its eigenfunctions (for more details about spherical harmonics see \cite{muller66}).

\begin{prop}
	The eigenfunctions of $S_k$ are the spherical harmonics.
\end{prop}

\begin{proof}
	The space of spherical harmonics of fixed degree is an irreducible representation of $SO(n)$. Hence, by Schur's lemma, it is enough to show the operator $S_k$ commutes with the $SO(n)$ action. Let $U\in SO(n)$, let $f\in L^2(S^{n-1})$ and let $g\in L^2(G_{n,k})$.
	We denote by $T_U$ and $T_U'$ the action of $SO(n)$ (the Koopman representation) on $L^2(S^{n-1})$ and $L^2(G_{n,k})$ respectively,
	\[T_Uf(x)=f(U^{-1}x),\;T_Ug(H)=g(U^{-1}H).\]
	By the definitions of the Radon transform and the measures, we have
	\begin{align*}
		T_U'\left(R_kf\right)(H)&=R_kf(U^{-1}H)=\int_{U^{-1}H\cap S^{n-1}}f(x)d\sigma_{U^{-1}H}(x)=\int_{H\cap S^{n-1}}f(U^{-1}x)d\sigma_{H}(x)\\
		&=R_k\left(T_Uf\right)(H).
	\end{align*}
	Since $R_k^*$ is conjugate to $R_k$ we have 
	\[T_U\left(R_k^*g\right)(\theta)=R_k^*\left(T_U'g\right)(\theta).\]
	Hence,
	\[T_U\left(S_kf\right)(\theta)=T_U\left(R_kR_k^*f\right)(\theta)=R_kR_k^*\left(T_Uf\right)(\theta)=S_k\left(T_Uf\right)(\theta).\]
\end{proof}

By considering the map $\psi:[-1,1]\times S^{n-2}\to S^{n-1}$ defined by $\psi(s,x)=(s,\sqrt{1-s^2}x)$, we obtain the following standard formula for integrating on the sphere,

\begin{prop}\label{prop:fubini}
	Let $f\in L^2\left(S^{n-1}\right)$. Then,
	\[\int_{S^{n-1}}f(x)d\sigma_{n-1}(x)=\tau_{n}\int_{-1}^1\left(\int_{S^{n-2}}
	f\left(t, \sqrt{1-t^2}y\right)d\sigma_{n-2}(y)\right)\left(1-t^2\right)^{(n-3)/2}dt,\]
	where $\tau_{n}=\Gamma(n/2)/(\sqrt{\pi}\Gamma(n/2-1/2))=\sqrt{n/(2\pi)}+O(1/\sqrt{n})$.
\end{prop}

Let $P_{\ell}$ denote the Gegenbauer polynomial of degree $\ell$ with respect to the weight function $(1-t^2)^{(n-3)/2}$. We use the standard normalization \cite{muller66},
\[P_{\ell}(1)=\binom{\ell+n-3}{\ell}.\]

It is well known (e.g \cite{muller66}) that we can define a spherical harmonic of degree $\ell$ by using the Gegenbauer polynomial of the same degree.
\begin{prop}
	Let $\xi\in S^{n-1}$, then $f_{\ell}(\theta)=P_{\ell}(\dprod{\theta}{\xi})$ is a spherical harmonic of degree $\ell$.
\end{prop}

Since the space of spherical harmonics of a fixed degree is irreducible, they all share the same eigenvalue for $S_k$.
Hence, we can combine both propositions, and express the eigenvalues of $S_k$ by one dimensional integrals of the Gegenbauer polynomials.
We note that when $\ell$ is odd and $f_{\ell}$ is a spherical harmonic of degree $\ell$, it is an odd function. 
Hence, $R_kf_{\ell}\equiv 0$, or equivalently, $\lambda_{k,\ell}=0$ for odd $\ell$. Therefore we need to consider only spherical harmonics of even degree.

\begin{prop}\label{prop:eigen}
	Let $\ell\geq 0$ be even. Let $\lambda_{k,\ell}^2$ be the eigenvalue of $S_k$ that corresponds to spherical harmonics of degree $\ell$. Then,
	\[\lambda_{k,\ell}^2=\frac{\tau_{k}\int_{-1}^{1}P_{\ell}(t)\left(1-t^2\right)^{(k-3)/2}dt}{\binom{\ell+n-3}{\ell}}.\]
\end{prop}

\begin{proof}
	We choose the spherical harmonic $f_{\ell}(x)=P_{\ell}\left(\dprod{x}{e_1}\right)$, where $e_1=(1,0,\ldots,0)$.
	By definition of $\lambda_{k,\ell}$ we have,
	\[\left(S_kf_{\ell}\right)(\theta)=\lambda_{k,\ell}^2f_{\ell}(\theta).\]
	By choosing $\theta=e_1$ and the normalization of the Gegenbauer polynomials, we have
	\[\left(S_kf_{\ell}\right)(e_1)=\lambda_{k,\ell}^2P_{\ell}(1)=\binom{\ell+n-3}{\ell}\lambda_{k,\ell}^2.\]
	Hence, we need to show that
	\[\left(S_kf_{\ell}\right)(e_1)=
	\tau_k\int_{-1}^{1}P_{\ell}(t)\left(1-t^2\right)^{(k-3)/2}dt.\]
	By definition of the operators $R_k$ and $R_k^*$ we have,
	\[\left(S_kf_{\ell}\right)(e_1)=\left(R_k^*R_kf_{\ell}\right)(e_1)=\int_{\{E;\;e_1\in E\}}\int_{S^{n-1}\cap E}f_{\ell}(x)d\sigma_E(x)d\mu_{k,e_1}(E).\]
	Using Proposition (\ref{prop:fubini}) on the inner integral, we have
	\[\left(S_kf_{\ell}\right)(e_1)=\tau_k\int_{\{E;\;e_1\in E\}}\int_{-1}^{1}\int_{S^{n-1}\cap E \cap e_1^{\perp}}f_{\ell}(t,\sqrt{1-t^2}y)d\sigma_{E \cap e_1^{\perp}}(y)\left(1-t^2\right)^{(k-3)/2}dtd\mu_{k,e_1}.\]
	By our choice of $f_{\ell}$ we have $f_{\ell}(t,y)=P_{\ell}(t)$.
	Hence, the integrals with respect to $\sigma_{E \cap e_1^{\perp}}$ and $\mu_{k,e_1}$ are on a constant function, and the proof is complete.
\end{proof}

In \cite{grupel17} we studied the random variable
\[\sqrt{\int_{S^{n-1}\cap H}fd\sigma_{H}\int_{S^{n-1}\cap H^{\perp}}fd\sigma_{H^{\perp}}},\]
where $\dim(H)=n/2$. Using a similar method as above we can calculate the second moment of this random variable.

\begin{theorem}\label{thm:correlation}
	Let $1\leq k\leq n-1$. For any $f,g\in L^2(S^{n-1})$ we have,
	\[\left|\int R_kf(H)R_{n-k}g(H^{\perp})-\int fd\sigma_{n-1}\int gd\sigma_{n-1}\right|\leq\sum_{j=1}^{\infty}\binom{\ell+n/2-2}{\ell}\binom{2\ell+n-3}{2\ell}^{-1}\norm{f_{2\ell}}\norm{g_{2\ell}},\]
	and
	\[\int R_kf(H)R_{n-k}f(H^{\perp})=\sum_{j=0}^{\infty}(-1)^{\ell}\binom{\ell+n/2-2}{\ell}\binom{2\ell+n-3}{2\ell}^{-1}\norm{f_{2\ell}}^2,\]
	where $f_{2\ell}$ and $g_{2\ell}$ are the projections of $f$ and $g$ to the space of spherical harmonics of degree $2\ell$.
\end{theorem}

\begin{proof}
	Let $\varphi:G_{n,k}\to G_{n,n-k}$ be the involution
	\[\varphi(H)=H^{\perp}.\]
	As before, the operator $R_k^*\varphi R_{n-k}$ commutes with rotations, hence it is diagonalized by the spherical harmonics. Let $P_{2\ell}$ be the Gegenbauer polynomial of degree $2\ell$ and $f_{2\ell}$ the spherical harmonics defined by $f_{2\ell}(x)=P_{2\ell}(x_1)$. Denoting by $\eta_{2\ell}$ the corresponding eigenvalue, we have
	\[\eta_{2\ell}P_{2\ell}(1)=\eta_{2\ell}f_{2\ell}(e_1)=R_k^*\varphi R_{n-k}f_{2\ell}(e_1)=\int_{\{H;\;e_1\in H\}}\int_{S^{n-1}\cap H^{\perp}}f_{2\ell}d\sigma_{H^{\perp}}d\mu_{e_1}.\]
	Since $e_1\in H$ for every $x\in H^{\perp}$ we have $\dprod{x}{e_1}=0$. Hence, the inner integral is on the constant function $f_{2\ell}(0)$, and we obtain
	\[\eta_{2\ell}P_{2\ell}(1)=f_{2\ell}(0)=P_{2\ell}(0).\]
	Hence, by using \cite[equation 4.7.31]{szego75}
	\[\eta_{2\ell}=\frac{P_{2\ell}(0)}{P_{2\ell}(1)}=(-1)^{\ell}\binom{\ell+n/2-2}{\ell}\binom{2\ell+n-3}{2\ell}^{-1}.\]
	Let $f$ and $g$ be a spherical harmonics of degree $2\ell$, we have
	\[\dprod{R_kf}{\varphi R_{n-k}g}=\dprod{f}{R_k^*\varphi R_{n-k}g}\leq|\eta_{2\ell}|\norm{f}\norm{g}.\]
	If $f=g$ then we can get equality
	\[\dprod{R_kf}{\varphi R_{n-k}f}=\eta_{2\ell}\norm{f}^2.\]
	In order to finish the proof, we remember the spherical harmonics of different degrees are orthogonal to each other.
\end{proof}

\begin{remark}
	Theorem \ref{thm:correlation} shows that the integral over the Grassmanian
	\[\int R_kf(H)R_{n-k}f(H^{\perp}),\]
	does not depend on the dimension of $H$.
\end{remark}

Note that $P_{2\ell}(0)/P_{2\ell}(1)$ are the singular values of the Radon transform $R_{n-1}$. Hence, we can repeat the analysis of Klartag and Regev for estimating the norm of the projections and the sum. We get,
\begin{cor}
	Let $f,g:S^{n-1}\to[0,\infty)$. Assume that $\int f=\int g=1$, then for all $1\leq k \leq n-1$
	\[\left|\int R_kf(H)R_{n-k}g(H^{\perp})-1\right|\leq C\frac{\log(2\norm{f}_{\infty})\log(2\norm{g}_{\infty})}{n}.\]
\end{cor}

\section{Random Geodesics on the Sphere}\label{sec:sphere_geodesics}
The goal of this section is to prove Theorem \ref{thm:sphere_geodesics}. We do this by calculating the singular values of the Radon transform $R_2$ and use them to bound the variance of the random variable $\sigma_{L}(A\cap L)/\sigma_{n-1}(A)$.

There are two natural ways to get a random geodesics on the sphere. The first is intersecting the sphere with a random $2$~dimensional subspace. The second is by choosing a uniform random point $X\in S^{n-1}$ and a uniform direction $Y\in T_XS^{n-1}$ of unit length, and define $L$ as the geodesic curve that starts at $X$ in the direction $Y$. We note that these two methods have the same distribution.

\begin{prop}\label{prop:2singular}
	The eigenvalues of $S_2$ are
	\[\lambda_{2,2\ell}^2=\binom{\ell+n/2-2}{\ell}^2\binom{2\ell+n-3}{2\ell}^{-1}.\]
\end{prop}

\begin{proof}
	By Proposition \ref{prop:eigen} we have
	\[\lambda_{k,\ell}^2=\frac{\tau_{k}\int_{-1}^{1}P_{\ell}(t)\left(1-t^2\right)^{(k-3)/2}dt}{\binom{\ell+n-3}{\ell}}.\]
	Using the change of variables $t=\cos\theta$, we obtain
	\[\lambda_{2,2\ell}^2=\frac{\tau_{2}\int_{0}^{\pi}P_{2\ell}(\cos \theta)d\theta}{\binom{2\ell+n-3}{2\ell}}.\]
	We have
	\[\tau_2=\left(\int_{-1}^1\frac{1}{\sqrt{1-s^2}}ds\right)^{-1}=1/\pi,\]
	and by the zero coefficient of the trigonometric polynomial $P_{2\ell}(\cos \theta)$ (see \cite[equation 4.9.19]{szego75}), we have
	\[\int_{0}^{\pi}P_{2\ell}(\cos \theta)d\theta=\pi \binom{\ell+n/2-2}{\ell}^2.\]
\end{proof}

Next we prove that the biggest eigenvalue of $S_2$ is $\lambda_{2,2}^2$.

\begin{prop}\label{prop:decreasing}
	The eigenvalues of $S_2$, $\{\lambda_{2,2\ell}^2\}$ is a decreasing sequence.
\end{prop}

\begin{proof}
	By Proposition \ref{prop:2singular}, we have
	\[\frac{\lambda_{2,2\ell+2}^2}{\lambda_{2,2\ell}^2}=\frac{\binom{\ell+n/2-1}{\ell+1}^2\binom{2\ell+n-3}{2\ell}}{\binom{2\ell+n-1}{2\ell+2}\binom{\ell+n/2-2}{\ell}^2}=\frac{(2\ell+1)(n+2\ell-2)}{(2\ell+2)(n+2\ell-1)}.\]
	Hence, for the sequence to be decreasing, we need the ratio to be at most one. This condition is equivalent to
	\[n+4\ell\geq 0,\]
	which is always satisfied.
\end{proof}

By Proposition \ref{prop:2singular} we have
\[\lambda_{2,2}^2=\frac{n-2}{2(n-1)}.\]
Combining this with Proposition \ref{prop:decreasing}, we can give an upper bound for the variance of a Radon transform of a function by the variance of the original function.

\begin{cor}\label{cor:var}
	Let $f\in L^2\left(S^{n-1}\right)$ such that $\int_{S^{n-1}}f(x)d\sigma_{n-1}=0$. Then
	\[\norm{R_2f}_{L^2(G_{n,2})}\leq \frac{1}{\sqrt{2}}\norm{f}_{L^2(S^{n-1})}.\]
\end{cor}

Using Corollary \ref{cor:var} we can prove Theorem \ref{thm:sphere_geodesics}.

\begin{proof}
	Define $f(x)=1_A(x)/\sigma_{n-1}(A)-1$ and $X=\sigma_{L}(A\cap L)/\sigma_{n-1}(A)$. By Corollary \ref{cor:var} we have
	\[\norm{R_2f}^2\leq\frac{1}{2}\norm{f}^2=\frac{1}{2}\left(\frac{1}{\sigma_{n-1}(A)}-1\right).\]
	On the other hand,
	\[\norm{R_2f}^2=\int_{G_{n,2}}\left(\int_{S^{n-1}\cap H}f(x)d\sigma_H(x)\right)^2d\mu_2=\var\left(X\right).\]
	By Markov's inequality
	\[\bP_{L}\left(\left|X-1\right|\geq t\right)\leq \frac{\var X}{t^2}\leq\frac{1-\sigma_{n-1}(A)}{2\sigma_{n-1}(A)t^2}.\]
	Setting $\sigma_{n-1}(A)=1/2$ and $t=1/2^{1/3}$ finishes the proof.
\end{proof}

\begin{remark}\label{rmk:sharp}
	By looking at the set $A=\{x\in S^{n-1};\; |x_1|\geq T\}$ where $T\approx c/\sqrt{n}$ is chosen such that $\sigma_{n-1}(A)=1/2$ we see that the probability bound in Theorem \ref{thm:sphere_geodesics} cannot be improved to an expression that decays with the dimension $n$. See Figure \ref{fig:sphere_simulation} for simulation results of this example.
\end{remark}

\begin{remark}
	As we can see in  Figure \ref{fig:sphere_simulation}, there are no small ball estimates on either side and there is no apparent information about the shape of the distribution.
\end{remark}

\begin{figure}
	\begin{subfigure}{.5\textwidth}
		\centering
		\includegraphics[width=.95\linewidth]{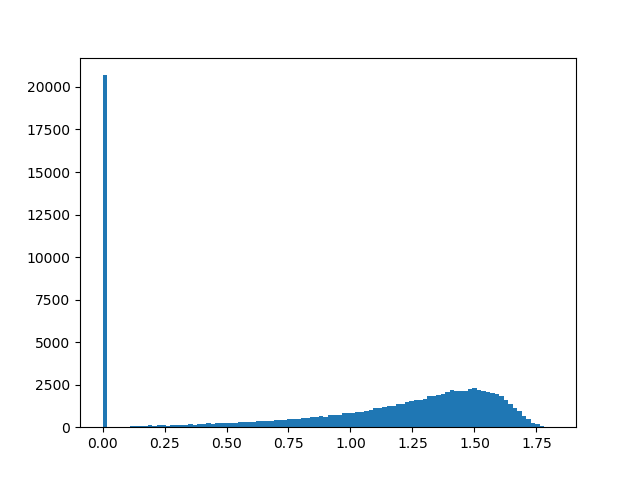}
		\caption{$A=\{x\in S^{n-1};\;|x_1|\geq c/\sqrt{n}\}$}
		\label{fig:sphere_two_caps}
	\end{subfigure}%
	\begin{subfigure}{.5\textwidth}
		\centering
		\includegraphics[width=.95\linewidth]{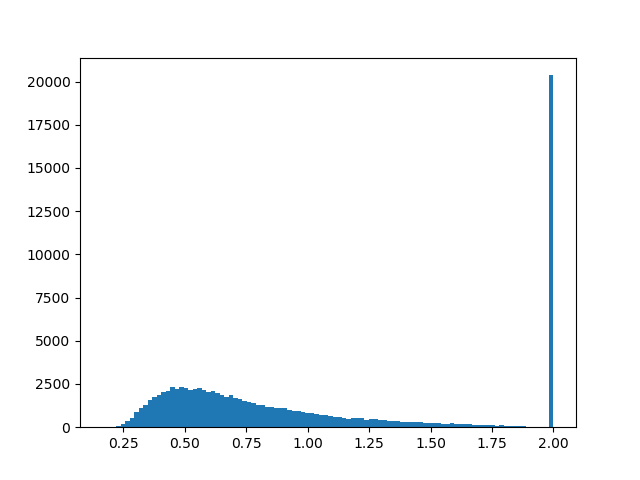}
		\caption{$A=\{x\in S^{n-1};\;|x_1|\leq c/\sqrt{n}\}$}
		\label{fig:sphere_central}
	\end{subfigure}
	\begin{subfigure}{.5\textwidth}
		\centering
		\includegraphics[width=.95\linewidth]{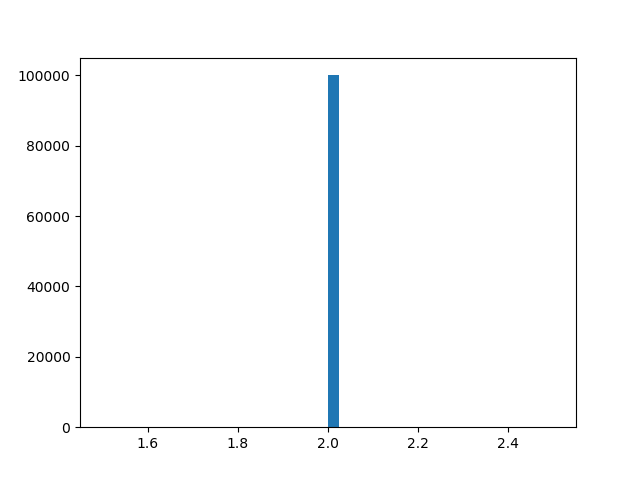}
		\caption{$A=\{x\in S^{n-1};\;x_1\geq 0\}$}
		\label{fig:sphere_cap}
	\end{subfigure}
	\caption{Histogram of simulations on the sphere in dimension $1000$ with $10^5$ samples for different sets $A$.}
	\label{fig:sphere_simulation}
\end{figure}

\section{Random Geodesics in a Convex Body}\label{sec:convex_geodesics}
In this section we prove that there is no analogous theorem to Theorem \ref{thm:sphere_geodesics} when we choose random geodesics inside a convex body.
We show that for every convex body, we can find a set such that the relative length of the intersection with a random line is close to a zero-one law.

There are different distributions for geodesics inside a convex body. We focus on the following model; choose a random point $X\in K$ distributed uniformly, and a random direction $\theta\in S^{n-1}$ distributed uniformly on the sphere and independent of $X$. The random geodesic is defined by $L=\{X+\bR \theta\}$.

This model of random geodesics is the {\it hit and run} model, commonly used in various algorithmic problems such as computing the volume of a convex body. 

We say a convex body $K\subseteq\bR^n$ is in \textit{isotropic position} if a random vector $X$ distributed uniformly inside $K$ has $\bE X=0$ and $\bE X\otimes X= Id$.
We start by two observations about isotropic convex bodies.

\begin{prop}
	Let $K\subseteq\bR^n$ be an isotropic convex body. Let $X$ be a random vector distributed uniformly in $K$. Then for any $a>0$ and any $\varepsilon>0$ we have
	\[\bP\left(|\dprod{X}{\xi}|\geq a+\varepsilon\right) \geq \bP\left(|\dprod{X}{\xi}|\geq a\right) - c\varepsilon,\]
	and
	\[\bP\left(|\dprod{X}{\xi}|\leq a-\varepsilon\right) \geq \bP\left(|\dprod{X}{\xi}|\leq a\right) - c\varepsilon,\]
	where $c>0$ is a universal constant and $\xi\in S^{n-1}$.
\end{prop}

\begin{proof}
	Let $f:\bR\to\bR_+$ be the density of $\dprod{X}{\xi}$, then $f$ is log-concave, $\int tf(t)dt=0$ and $\int t^2f(t)dt=1$.
	By \cite{fradelizi97} we have
	\[\sup f(t) \leq ef(0).\]
	By the Berwald-Borell lemma \cite{berwald47,borell74} the functions $M^+(p)=(\Gamma(p+1))^{-1}\int_0^{\infty}t^pf(t)dt$ and $M^-(p)=(\Gamma(p+1))^{-1}\int_{-\infty}^0t^pf(t)dt$ are log concave in $[-1,\infty)$. Hence,
	\[M^{\pm}(0)\geq \left(M^{\pm}(-1)\right)^{2/3}\left(M^{\pm}(2)\right)^{1/3}.\] 
	Since $f$ is a density function for an isotropic random variable, we have
	\begin{align*}
	&M^+(-1)=M^-(-1)=f(0)\\
	&M^+(0)+M^-(0)=1\\
	&M^+(2)+M^-(2)=1/2.
	\end{align*}
	Adding everything, and using $a^{1/3}+b^{1/3}\geq (a+b)^{1/3}$ for all $a,b\geq 0$, we have
	\[1=M^+(0)+M^-(0)\geq f^{2/3}(0)\left(\left(M^+(2)\right)^{1/3}+\left(M^-(2)\right)^{1/3}\right)\geq f^{2/3}(0)2^{-1/3}.\]
	Hence,
	\[\sup f\leq ef(0)\leq e\sqrt{2}.\]
	To conclude, we have
	\begin{align*}
	\bP\left(|\dprod{X}{\xi}|\geq a+\varepsilon\right) &= \bP\left(|\dprod{X}{\xi}|\geq a\right) - \bP\left(a\leq |\dprod{X}{\xi}|\leq a+\varepsilon\right)\\
	&\geq \bP\left(|\dprod{X}{\xi}|\geq a\right) -2\varepsilon\sup f\geq \bP\left(|\dprod{X}{\xi}|\geq a\right) -c\varepsilon.
	\end{align*}
	The second statement follows similarly.
\end{proof}

\begin{prop}\label{prop:length}
	Let $K\subseteq\bR^n$ be an isotropic convex body. Let $X$ be a random vector distributed uniformly in $K$. Let $\theta \in S^{n-1}$ and let $\ell:K\to \bR_+$ be the length function in direction $\theta$, that is $\ell(x)=\textrm{length}\left(K\cap \{x+\bR\theta\}\right)$. Then
	\[\bP(\ell(X)\geq t)\leq 2e^{-ct},\]
	where $c>0$ is a universal constant.
\end{prop}

\begin{proof}
	The function $\ell(x)$ depends only on the projection of $x$ to the hyperplane $\theta^{\perp}$. Hence, by Fubini's theorem we have,
	\[\bE \ell^p(X)=\frac{1}{\vol{K}}\int_{\proj_{\theta^{\perp}}K}\ell^p(x)\ell(x)dx.\]
	Let $S_{\theta}$ be the Steiner symmetrization in the $\theta$ direction and set $T=S_{\theta}K$ (for more details on the Steiner symmetrization see \cite{bonnesen+fenchel87}). Since the Steiner symmetrization is in the $\theta$ direction the function $\ell$ and $\proj_{\theta^{\perp}}K$ are preserved under it. Hence, for a random vectors $X$ and $Y$ distributed uniformly on $K$ and $T$ we have,
	\[\bE\ell^p(Y)=\bE\ell^p(X).\]
	In addition
	\[\var \dprod{Y}{\theta}\leq \var \dprod{X}{\theta}=1.\]
	Hence, by the Berwald-Borell lemma \cite{berwald47,borell74} we have
	\[\left(\bE |\dprod{Y}{\theta}|^p\right)^{1/p}\leq Cp \left(\bE |\dprod{Y}{\theta}|^2\right)^{1/2} \leq Cp.\]
	Since $\dprod{Y}{\theta}$ is a symmetric random variable, we have
	\begin{align*}
	\bE  |\dprod{Y}{\theta}|^p &= \frac{1}{\vol{T}}\int_{\proj_{\theta^{\perp}}T}\left(\int_{-\ell(y)/2}^{\ell(y)/2}|t|^pdt\right)\ell(y)dy=\frac{1}{\vol{K}}\int_{\proj_{\theta^{\perp}}K}\left(\int_{-\ell(x)/2}^{\ell(x)/2}|t|^pdt\right)\ell(x)dx\\
	&=\frac{1}{2^p(p+1)\vol{K}}\int_{\proj_{\theta^{\perp}}K}\ell^{p+1}(x)\ell(x)dx
	=\frac{1}{2^p(p+1)}\bE\ell^{p+1}(X)
	\end{align*}
	We have,
	\[\bE\ell^p(X)=2^{p-1}p\bE|\dprod{Y}{\theta}|^{p-1}\leq (2C)^{p-1}p^p\leq C_1^pp!.\]
	By Markov's inequality, for any $\alpha>0$ we have
	\[\bP(\ell(X)\geq t)=\bP(e^{\alpha\ell(X)}\geq e^{\alpha t})\leq e^{-\alpha t}\bE e^{\alpha\ell(X)}= e^{-\alpha t} \sum_{p=0}^{\infty} \frac{\alpha^p}{p!}\bE\ell^p(X)\leq e^{-\alpha t}\sum_{p=0}^{\infty}(C_1\alpha)^p.\]
	Choosing $\alpha=1/(2C_1)$ concludes the proof.	
\end{proof}

Another observation we use is the concentration of measure on the sphere. The following proposition can be proved by a direct computation or by the concentration of Lipschitz functions on the sphere.

\begin{prop}
	Let $\theta\in S^{n-1}$ be a random vector chosen by the uniform distribution, and let $\xi\in S^{n-1}$ be a fixed direction. Then
	\[\bP\left(|\dprod{\theta}{\xi}|\geq t\right)\leq Ce^{-nt^2/2},\quad\forall t>0.\] 
\end{prop}

We are now ready to show that random geodesic sampling in any isotropic convex body can be close to a zero one law. Let $K\subseteq\bR^n$ be a convex body. For any $\xi\in S^{n-1}$ define the set $A_{\xi}=\{x\in K;\;|\dprod{x}{\xi}|\geq t_{\xi}\}$, where $t_{\xi}>0$ is chosen such that $\vol{A_{\xi}}/\vol{K}=1/2$.

\begin{prop}
	Let $K\subseteq\bR^n$ be an isotropic convex body. Let $X$ be a random vector distributed uniformly in $K$. Let $L=\{X+\bR \theta\}$ where $X$ is uniform in $K$ and $\theta$ is uniform in $S^{n-1}$. For any $\xi\in S^{n-1}$ we have,
	\[\bP\left(\frac{\textrm{length}(L\cap A_{\xi})}{\textrm{length}(L\cap K)}\in\{0,1\}\right)=1-O^*\left(\frac{1}{\sqrt{n}}\right).\]
\end{prop}

\begin{proof}
	We define the following events:
	\begin{align*}
	D &= \left\{|\dprod{X}{\xi}|\geq t_{\xi} + c^{-1}\frac{\log^3n}{\sqrt{n}}\right\},\\
	P &= \left\{|\dprod{\theta}{\xi}| \leq \frac{\log n}{\sqrt{n}} \right\},\\
	M &= \left\{\textrm{length}(L\cap K) \leq c^{-1}\log n\right\},
	\end{align*}
	where $c>0$ is the constant from Proposition \ref{prop:length}.
	Assuming the event $D\cap P\cap M$, we have $X$ inside $A_{\xi}$ with distant to the inner boundary of $A_{\xi}$ of at least $c^{-1}\log^3n/\sqrt{n}$. In addition, the line passing through $X$ can diverge in the $\xi$ direction by at most $c^{-1}\log n\cdot \log n/\sqrt{n}$, hence it cannot escape it. Therefore,
	\[\bP\left(\frac{\textrm{length}(L\cap A_{\xi})}{\textrm{length}(L\cap K)}=1\right)\geq \bP(D\cap P\cap M).\]
	By previous propositions
	\[\bP(P)\geq 1-\frac{C}{n},\]
	\[\bP(M)\geq 1 - \frac{2}{n},\]
	and
	\[\bP(D)= \frac{1}{2} - C'\frac{\log^3n}{\sqrt{n}}.\]
	Hence,
	\[\bP\left(\frac{\textrm{length}(L\cap A_{\xi})}{\textrm{length}(L\cap K)}=1\right)\geq \frac{1}{2} - O^*\left(\frac{1}{\sqrt{n}}\right).\]
	We can repeat the argument by replacing $D$ with
	\[D'=\left\{|\dprod{X}{\xi}|\leq t_{\xi} - c^{-1}\frac{\log^3n}{\sqrt{n}}\right\},\]
	and obtain
	\[\bP\left(\frac{\textrm{length}(L\cap A_{\xi})}{\textrm{length}(L\cap K)}=0\right)\geq \frac{1}{2} - O^*\left(\frac{1}{\sqrt{n}}\right).\]
\end{proof}

\begin{remark}
	There are other examples of subsets with a similar property. For example, similar analysis shows that in an isotropic convex body $K$ with constant thin shell width (e.g the cube or the simplex) the set $A=K\cap(RB_2^n)$, where $B_2^n$ is the euclidean ball and $R$ is chosen such that $\vol{A}/\vol{K}=1/2$, has a similar property (when replacing $0,1$ with $o(1),1-o(1)$).
\end{remark}

In order to complete the proof of Theorem \ref{thm:convex_geodesics}, we note that we can repeat this argument for any convex body, but in a specific direction $\xi$. Let $K\subseteq \bR^n$ be a convex body. Let $X$ be a random vector distributed uniformly in $K$. Denote the covariance matrix of $X$ by $C=\bE X\otimes X$. Choosing $\xi = \textrm{argmax}\{\dprod{Cx}{x};\;x\in S^{n-1}\}$, will obtain the desired result.

\begin{figure}
	\begin{subfigure}{.5\textwidth}
		\centering
		\includegraphics[width=.95\linewidth]{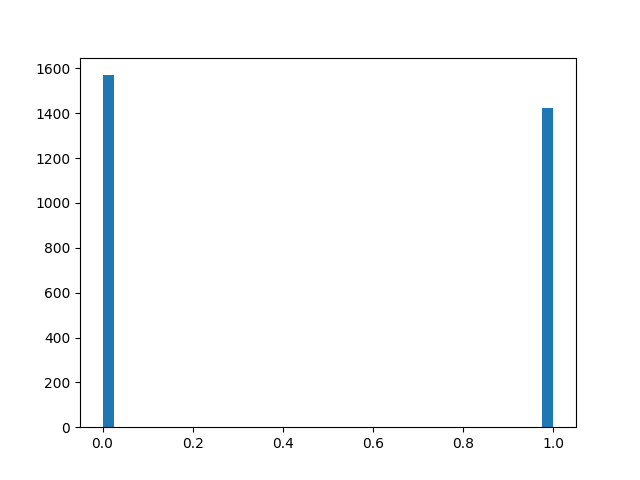}
		\caption{Standard simplex.}
		\label{fig:simplex_samples}
	\end{subfigure}%
	\begin{subfigure}{.5\textwidth}
		\centering
		\includegraphics[width=.95\linewidth]{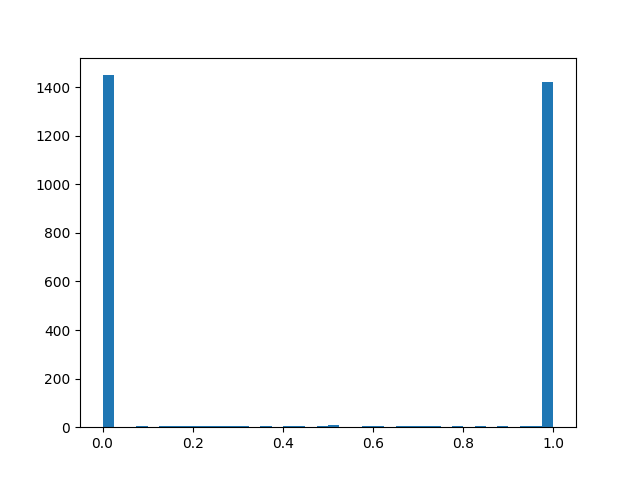}
		\caption{Ball}
		\label{fig:ball_samples}
	\end{subfigure}
	\caption{Histogram of simulations on the sphere in dimension $1000$ with $3000$ samples for $A_{\xi}$ defined on different convex bodies.}
	\label{fig:convex_simulation}
\end{figure}

While Theorem \ref{thm:convex_geodesics} shows that for some subsets of the simplex, intersection with random geodesic will have a zero one law, Theorem \ref{thm:sphere_geodesics} shows that for the simplex there is another curve model that samples subsets of the simplex in a more representative way.

Let $\Delta_n=\{x\in\bR^n;\;x_i\geq 0,\; \sum x_i \leq 1\}$ be the Archimedes simplex in $\bR^n$. It is well known that the transformation $\pi_n:S^{2n+1}\to \Delta_n$ defined by
\[\pi_n(x)=(x_1^2+x_{n+2}^2,\ldots,x_{n}^2+x_{2n+1}^2)\]
pushes the surface area measure of the sphere to the volume measure of the simplex.
Let $x,y\in S^{2n+1}$ be orthogonal to each other, and let $\gamma(t)=x\cos t+y\sin t$, be the geodesic curve that starts at $x$ in the direction $y$, and $\delta=\pi_n\circ\gamma$. We have

\begin{align*}
	\left(\delta(t)\right)_j&=\left((\pi_n\gamma)(t)\right)_j=(c_j\cos t+y_j\sin t)^2+(c_{n+j+1}\cos t+y_{n+j+1}\sin t)^2\\
	&=\frac{1}{2}(x_j^2+x_{n+j+1}^2+y_j^2+y_{n+j+1}^2) + \frac{1}{2}(x_j^2+x_{n+j+1}^2-y_j^2-y_{n+j+1}^2)\cos(2t)\\
	& + (x_jy_j+x_{n+j+1}y_{n+j+1})\sin(2t).
\end{align*}

Hence, the map $\pi_n$ maps geodesics on a sphere to a ellipses in $\Delta_n$.
Hence, we can push forward our result to the simplex and get,

\begin{cor}\label{cor:simplex}
	Let $A\subseteq\Delta_n$ be with $\vol{A}/\vol{\Delta_n}=1/2$. Let $L$ be a random ellipse in $\Delta_n$ chosen as the $\pi_n$ image of a uniform geodesic on $S^{2n-1}$. We have,
	\[\bP\left(\left|\frac{2\mu(A\cap L)}{\mu(L)}-1\right|\geq \frac{1}{2^{1/3}}\right)\leq \frac{1}{2^{1/3}},\]
	where $\mu$ is the push forward on the length by the map $\pi_n$.
\end{cor}

\begin{question}
	Let $K\subseteq\bR^n$ be a convex body, can we find a distribution of curves inside $K$ that would generalize Corollary \ref{cor:simplex} to $K$?
\end{question}

\section{Arithmetic Progressions in Finite Fields}\label{sec:finite_fields}
In this section we demonstrate how a similar technique to the one we used to prove Theorem~\ref{thm:sphere_geodesics}, can be applied in other settings. We analyze the intersection of subsets of the $n$~dimensional discrete torus with random arithmetic progressions.

Let $p>0$ be a prime number and let $n>0$ be an integer. Denote $\bZ_p=\bZ/p\bZ$, and 
\[G_p=\{(a, a+b,\ldots,a+(p-1)b);\; a,b\in\bZ_p^n,\;b\ne 0\}.\] 
For convenience we shall denote a sequence $(a, a+b,\ldots,a+(p-1)b)$ by the pair $(a;b)$.
We define a discrete version of the Radon transform $R:L^2(\bZ_p^n)\to L^2(G_p)$ by
\[(Rf)(a;b)=\frac{1}{p}\sum_{j=0}^{p-1}f(a+jb).\]
The conjugate transform $R^*:L^2(G_p)\to L^2(\bZ_p^n)$ is defined by
\[(R^*g)(a)=\frac{1}{p^n-1}\sum_{b\ne 0}g(a;b).\]
As before, we define $S=R^*R$.

\begin{prop}
	The eigenvalues of the operator $S$ are
	\[\lambda^2 = \frac{p^{n-1}-1}{p^n-1},\]
	for functions with mean zero, and
	\[\lambda^2=1,\]
	for constant functions.

\end{prop}

\begin{proof}
	The operator $S$ commutes with translations. Hence, the eigenfunctions of $S$ are of the form $f_y(x)=e^{i2\pi\dprod{x}{y}/p}$ for $y\in \bZ_p^n$.
	We have,
	\[Sf_y(a)=\frac{1}{p(p^{n}-1)}\sum_{b\ne 0}\sum_{j=0}^{p-1} f_y(a+jb).\]
	By the definition of $f_y$ we have
	\[\sum_{j=0}^{p-1}f_y(a+jb)=\sum_{j=0}^{p-1}e^{i2\pi \dprod{a+jb}{y}}=e^{i2\pi \dprod{a}{y}/p}\sum_{j=0}^{p-1}e^{i2\pi j \dprod{b}{y}/p}=f_y(a)\begin{cases}
	p, & \dprod{b}{y}=0\\
	0, & \dprod{b}{y}\ne 0
	\end{cases}.\]
	Hence,
	\[S(f_y)(a)=f_y(a)\frac{1}{p^{n}-1}\#\{b;\;\dprod{b}{y}=0, b\ne 0\}=f_y(a)\frac{1}{p^{n}-1}\begin{cases}
	p^n-1,\; y=0\\
	p^{n-1}-1,\; y\ne 0
	\end{cases}.\]
	We have,
	\[\lambda_y^2 = \frac{p^{n-1}-1}{p^n-1},\quad\forall y\ne 0.\]
\end{proof}

Using the singular values of the Radon transform, we can bound the variance as before, and get a result analogous to Theorem \ref{thm:sphere_geodesics}.

\begin{theorem}
	Let $A\subseteq \bZ^n$ with $m(A)=1/2$ and let $a,b\in \bZ_p^n$ be random vectors such that $b\ne0$. Let $L=\{a, a+b,\ldots,a+(p-1)b\}$ be the arithmetic progression defined by $a$ and $b$. Then,
	\[\bP\left(\left|\#(A\cap L)-\frac{p}{2}\right|\geq \sqrt{\frac{p}{2}}\right)\leq \frac{1}{2}.\]
\end{theorem}  

\section{Intersection With Higher Dimensional Subspaces}\label{sec:general_k}
In \S~\ref{sec:radon} we saw how to express the singular values of the $k$~dimensional Radon transform $R_k$ by one dimensional integration of the Gegenbauer polynomials with respect to some weight function. We used this result in \S~\ref{sec:sphere_geodesics} in order to find the behavior of intersection with random geodesics on the sphere through the special case $k=2$. The purpose of this section is to understand random intersection on the sphere by subspaces of higher dimension $k$ by studying the singular values of $R_k$ for $k>2$.

It is possible to generalize the approach of Section \ref{sec:sphere_geodesics} and use the known coefficients of the trigonometric polynomials $P_{\ell}(\cos t)$ (see \cite{szego75}) and the Fourier series of $\sin^{k-2}(t)$. By orthogonality of the Fourier basis, this will be a finite sum, and the number of summands in $\lambda_{k,2\ell}$ is $\min\{\ell,k/2-1\}$. If either $k$ or $\ell$ are small, the calculations are simple, but when  both $k$ and $\ell$ are large, it becomes difficult to handle.

We begin with an example of results of this method. Here, we fix a small $k$ and calculate $\lambda_{k,2\ell}^2$ for all $\ell$.

\begin{example}
	The eigenvalues of $S_4$ are 
	\[\lambda_{4,\ell}^2=\binom{\ell+n-3}{\ell}^{-1}\binom{\ell/2+n/2-2}{\ell/2}^2\frac{2n-8}{(\ell+n-4)(\ell+2)}.\]
	We obtain $|\lambda_{4,2}|\leq 1/2$. For large $\ell$ we have
	\[\lambda_{4,\ell}^2\approx C\left(\frac{n}{\ell(n+\ell)}\right)^{3/2}.\]
	In addition, using Proposition \ref{prop:decreasing} the sequence $\{\lambda_{4,2\ell}^2\}$ is a product of two positive decreasing sequences, hence it is also decreasing.
	We obtain, that for any  $f\in L^2\left(S^{n-1}\right)$ such that $\int_{S^{n-1}}f(x)d\sigma_{n-1}=0$,
	\[\norm{R_4f}_{L^2(G_{n,2})}\leq \frac{1}{2}\norm{f}_{L^2(S^{n-1})}.\]
\end{example}

Similarly we can calculate the eigenvalue $\lambda_{k,2}^2$ for all even $k\geq 2$.

\begin{example}
	For any even $k\geq 4$ we have,
	\[\lambda_{k,2}^2=\tau_k\frac{\pi(n-2)}{2^{k-1}}\binom{k-2}{k/2-1}\frac{n-k}{k}\binom{n-1}{2}^{-1}=\frac{n-k}{k(n-1)}.\]
\end{example}

We employ a different technique in order to calculate the eigenvalues of $S_k$ for all $k$ and $\ell$.

\begin{prop}\label{prop:general_singular_values}
	Let $2\leq k\leq n$ and let $\ell\geq 1$. Then,
	\[\lambda_{k,2\ell}^2=\tau_k \binom{2\ell+n-3}{2\ell}^{-1}\frac{\sqrt{\pi}\Gamma(2\ell+n/2-1)\Gamma(k/2-1/2)}{\Gamma(\ell+1)\Gamma(k/2+\ell)\Gamma(n/2-1)}\frac{\Gamma(n/2+l-1)\Gamma(n/2+l-k/2+1/2)}{\Gamma(n/2-k/2)\Gamma(n/2+2l-1)}.\]
\end{prop}

\begin{proof}
	By Proposition \ref{prop:eigen}, it is enough to calculate
	$\int_{-1}^1 P_{2\ell}(t)(1-t^2)^{k/2-3/2}dt$.
	By \cite[equation 4.7.31]{szego75} we have
	\[P_{2\ell}(t)=\sum_{j=0}^{\ell}(-1)^j\frac{2^{2\ell-2j}\Gamma(2\ell-j+n/2-1)}{\Gamma(n/2-1)\Gamma(j+1)\Gamma(2\ell-2j+1)}t^{2\ell-2j}.\]
	Hence, we can start by integrating only the monomials. By standard computations,
	\[\int_{-1}^1t^{2m}(1-t^2)^{k/2-3/2}dt=\frac{\Gamma(m+1/2)\Gamma(k/2-1/2)}{\Gamma(m+k/2)}.\]
	Combining the two, we have
	\[\int_{-1}^1 P_{2\ell}(t)(1-t^2)^{k/2-3/2}dt=\frac{\Gamma(k/2-1/2)}{\Gamma(n/2-1)}\sum_{j=0}^{\ell}(-1)^j\frac{2^{2\ell-2j}\Gamma(2\ell-j+n/2-1)\Gamma(\ell-j+1/2)}{\Gamma(j+1)\Gamma(2\ell-2j+1)\Gamma(\ell-j+k/2)}.\]
	Multiplying inside the sum and diving outside the sum by
	\[\frac{\Gamma(\ell+1)\Gamma(k/2+\ell)}{\sqrt{\pi}\Gamma(2\ell+n/2-1)},\]
	we need to sum over
	\[(-1)^j\frac{2^{2\ell-2j}\Gamma(2\ell-j+n/2-1)\Gamma(\ell-j+1/2)\Gamma(\ell+1)\Gamma(k/2+\ell)}{\Gamma(j+1)\Gamma(2\ell-2j+1)\Gamma(\ell-j+k/2)\sqrt{\pi}\Gamma(2\ell+n/2-1)}.\]
	Since (see \cite[equation 6.1.18]{abramowitz+stegun64})
	\[\frac{2^{2\ell-2j}\Gamma(\ell-j+1/2)}{\sqrt{\pi}\Gamma(2\ell-2j+1)}=\frac{1}{\Gamma(\ell-j+1)},\]
	we can simplify the expression. We denote by $(a)_b$ the Pochhammer symbol. When $a,b\geq 0$ and $b\in\bZ$, we have $(a)_b=\Gamma(a+b)/\Gamma(a)$ and $(-a)_b=(-1)^b\Gamma(a+1)/\Gamma(a-b+1)$. Using these notations, we can reduce the summands to
	\[(-1)^j\frac{\Gamma(\ell+1)\Gamma(k/2+\ell)\Gamma(2\ell+n/2-j-1)}{\Gamma(j+1)\Gamma(\ell-j+1)\Gamma(k/2+\ell-j)\Gamma(2\ell+n/2-1)}=(-1)^j\binom{\ell}{j}\frac{(-k/2-\ell+1)_j}{(-2\ell-n/2+2)_j}.\]
	By the Vandermonde identity (see \cite[Appendix III]{slater66} or the next proposition)
	\[\sum_{j=0}^{\ell}(-1)^j\binom{\ell}{j}\frac{(-k/2-\ell+1)_j}{(-2\ell-n/2+2)_j}=\frac{\Gamma(n/2+l-1)\Gamma(n/2+l-k/2+1/2)}{\Gamma(n/2-k/2)\Gamma(n/2+2l-1)}.\]
	Hence,
	\[\int_{-1}^1 P_{2\ell}(t)(1-t^2)^{k/2-3/2}dt=\frac{\sqrt{\pi}\Gamma(2\ell+n/2-1)\Gamma(k/2-1/2)}{\Gamma(\ell+1)\Gamma(k/2+\ell)\Gamma(n/2-1)}\frac{\Gamma(n/2+l-1)\Gamma(n/2+l-k/2+1/2)}{\Gamma(n/2-k/2)\Gamma(n/2+2l-1)}.\]
\end{proof}

In the proof of Proposition \ref{prop:general_singular_values} we used the Vandermonde identity with integer or half integer parameters. While the identity holds true with greater generality (see \cite{slater66}), when both $k/2$ and $n/2$ are integers, we can prove it using a simple double counting argument. In order to prove this, we write the Pochhammer symbols as factorials, then the special case of the Vandermonde identity we use is equivalent to the following proposition.

\begin{prop}
	Let $a,b,c>0$ be integers, and assume that $c\geq a$. Then,
	\[\sum_{j=0}^b (-1)^j\binom{b}{j}\frac{(a+b)!(2b+c-j)!}{(a+b-j)!(2b+c)!}=\frac{(b+c)!(c-a+b)!}{(c-a)!(2b+c)!}.\]
\end{prop}

\begin{proof}
	Assume we have $b$ red balls, and $c+b$ green balls. We compute the probability to choose $a+b$ green balls out of the $c+2b$ balls.
	
	On the one hand, we can count the ways to have $a+b$ balls out of the $c+b$ green balls and divide by the number of choices of $a+b$ balls,
	\[\binom{c+b}{a+b}\binom{2b+c}{a+b}^{-1}=\frac{(b+c)!(c-a)!}{(2b+c)!(b+c-a)!}.\]
	
	On the other hand, we can use the inclusion-exclusion principle with the events $A_k$ of having the $k$-th red ball. Then,
	\begin{align*}
		\bP(\textrm{having no apples})&=1-\bP(\textrm{having at least on apple})=1-\bP(A_1\cup\cdots\cup A_b)\\
		&=1-\left(\bP(A_1)+\cdots+\bP(A_b)\right)+\left(\bP(A_1\cap A_2)+\cdots+\bP(A_{b-1}\cap A_b)\right)\\
		&+\cdots+(-1)^b\bP(A_1\cap\cdots\cap A_b).
	\end{align*}	
	Using the symmetry of index permutations, we have
	\begin{align*}
		\bP(\textrm{having no apples}) &=\sum_{j=0}^b(-1)^j\binom{b}{j}\bP\left(\bigcap_{i=1}^jA_i\right)=\sum_{j=0}^b(-1)^j\binom{b}{j}\binom{2b+c-j}{a+b-j}\binom{2b+c}{a+b}^{-1}\\
		&=\sum_{j=0}^b (-1)^j\binom{b}{j}\frac{(a+b)!(2b+c-j)!}{(2b+c)!(a+b-j)!}.
	\end{align*}
\end{proof}

\begin{cor}\label{cor:decreasing}
	Let $2\leq k\leq n-1$. The sequence $\{\lambda_{k,2\ell}\}^2$ of eigenvalues of $S_k$ is a decreasing sequence.
\end{cor}

\begin{proof}
	Using the functional relationships $\Gamma(x+1)/\Gamma(x)=x$ and $(x+1)_m/(x)_m=(m+x)/x$, by Proposition \ref{prop:general_singular_values}, we have
	\begin{equation}\label{eq:decreasing}
		\frac{\lambda_{k,2\ell+2}^2}{\lambda_{k,2\ell}^2}=\frac{(2\ell+1)(2\ell+n-k)}{(2\ell+k)(2\ell+n-1)}.\tag{$\star$}
	\end{equation}
	We note that for any $k>1$, this ratio is strictly less than one, hence the sequence is decreasing.
\end{proof}

Using the above calculations, we can prove Theorem \ref{thm:general_k}.
\begin{proof}[Proof of Theorem \ref{thm:general_k}]
	By Corollary \ref{cor:decreasing} the non trivial singular values of $R_k$ are at most $\lambda_{k,2}$. By Proposition \ref{prop:general_singular_values} we have
	\[\lambda_{k,2}^2=\tau_k \binom{n-1}{2}^{-1}\frac{\sqrt{\pi}\Gamma(n/2+1)\Gamma(k/2-1/2)}{\Gamma(k/2+1)\Gamma(n/2-1)}\frac{n-k}{n}.\]
	In addition,
	\[\tau_k=\frac{\Gamma(k/2)}{\sqrt{\pi}\Gamma(k/2-1/2)}.\]
	We have,
	\[\lambda_{k,2}^2=\frac{2(n-k)\Gamma(n/2+1)\Gamma(k/2)}{n(n-1)(n-2)\Gamma(k/2+1)\Gamma(n/2-1)}=\frac{n-k}{k(n-1)}.\]
	Hence, for any $f\in L^2(S^{n-1})$ we have
	\[\var(R_kf)\leq \frac{n-k}{k(n-1)}\var(f).\]
	To finish the proof, we set $f$ to be the normalized indicator of the subset $A\subseteq S^{n-1}$.
\end{proof}

The importance of Corollary \ref{cor:decreasing} is not only in showing that the sequence is decreasing, but also by the rate of the decrease in \eqref{eq:decreasing}.

\begin{example}
	When $k=n-1$, Equation \eqref{eq:decreasing} gives us
	\[\lambda_{n-1,2\ell}\leq \left(C\frac{\ell}{n+\ell}\right)^{\ell}.\]
	This is the rate used by Klartag and Regev in \cite{klartag+regev11} to prove the result for intersection with hyper-planes.
\end{example}

Another interesting case is when $k=\lfloor n/2\rfloor$. Equation \eqref{eq:decreasing} gives us
\[\lambda_{n/2,2\ell}\leq \left(C\frac{\ell}{n+\ell}\right)^{\ell/2}.\]

An interesting question would be to give a direct proof for Theorem 6.1 in \cite{klartag+regev11} using this result. 
A possible starting point would be to generalize the hypercontractivity result (Lemma 5.3 in \cite{klartag+regev11}) for the Grassmanian manifold $G_{n,k}$.


\end{document}